\documentclass[twoside,leqno,10pt, A4]{amsart}
\usepackage{amsfonts}
\usepackage{amsmath}
\usepackage{amscd}
\usepackage{amssymb}
\usepackage{amsthm}
\usepackage{amsrefs}
\usepackage{latexsym}
\usepackage{mathrsfs}
\usepackage{bbm}
\usepackage{enumerate}
\usepackage{color}
\begin{document}

\newtheorem{theorem}[subsection]{Theorem}
\newtheorem{proposition}[subsection]{Proposition}
\newtheorem{lemma}[subsection]{Lemma}
\newtheorem{corollary}[subsection]{Corollary}
\newtheorem{conjecture}[subsection]{Conjecture}
\newtheorem{prop}[subsection]{Proposition}
\numberwithin{equation}{section}
\newcommand{\mr}{\ensuremath{\mathbb R}}
\newcommand{\mc}{\ensuremath{\mathbb C}}
\newcommand{\dif}{\mathrm{d}}
\newcommand{\intz}{\mathbb{Z}}
\newcommand{\ratq}{\mathbb{Q}}
\newcommand{\natn}{\mathbb{N}}
\newcommand{\comc}{\mathbb{C}}
\newcommand{\rear}{\mathbb{R}}
\newcommand{\prip}{\mathbb{P}}
\newcommand{\uph}{\mathbb{H}}
\newcommand{\fief}{\mathbb{F}}
\newcommand{\majorarc}{\mathfrak{M}}
\newcommand{\minorarc}{\mathfrak{m}}
\newcommand{\sings}{\mathfrak{S}}
\newcommand{\fA}{\ensuremath{\mathfrak A}}
\newcommand{\mn}{\ensuremath{\mathbb N}}
\newcommand{\mq}{\ensuremath{\mathbb Q}}
\newcommand{\half}{\tfrac{1}{2}}
\newcommand{\f}{f\times \chi}
\newcommand{\summ}{\mathop{{\sum}^{\star}}}
\newcommand{\chiq}{\chi \bmod q}
\newcommand{\chidb}{\chi \bmod db}
\newcommand{\chid}{\chi \bmod d}
\newcommand{\sym}{\text{sym}^2}
\newcommand{\hhalf}{\tfrac{1}{2}}
\newcommand{\sumstar}{\sideset{}{^*}\sum}
\newcommand{\sumprime}{\sideset{}{'}\sum}
\newcommand{\sumprimeprime}{\sideset{}{''}\sum}
\newcommand{\shortmod}{\ensuremath{\negthickspace \negthickspace \negthickspace \pmod}}
\newcommand{\V}{V\left(\frac{nm}{q^2}\right)}
\newcommand{\sumi}{\mathop{{\sum}^{\dagger}}}
\newcommand{\mz}{\ensuremath{\mathbb Z}}
\newcommand{\leg}[2]{\left(\frac{#1}{#2}\right)}
\newcommand{\muK}{\mu_{\omega}}

\title[Moments of central values of quartic Dirichlet $L$-functions]{Moments of central values of quartic Dirichlet $L$-functions}

\author{Peng Gao and Liangyi Zhao}

\begin{abstract}
In this paper, we study moments of the central values of quartic Dirichlet $L$-functions and establish quantitative non-vanishing result for these $L$-values.
\end{abstract}

\maketitle



\section{Introduction}

 Moments of $L$-functions at the central point over a family of characters of a fixed order have been extensively studied in the literature. The first and second moments of quadratic Dirichlet $L$-functions were evaluated by M. Jutila \cite{Jutila}.   The error term for the first moment in \cite{Jutila} was improved in \cites{DoHo, MPY, ViTa}. Asymptotic formulas for the second and third moments for the same family with power savings in the error terms were obtained by K. Soundararajan \cite{sound1}.  The error term for the third moment was subsequently improved in \cite{DGH} and \cite{Young1}.  In the number field case, the authors studied in \cites{G&Zhao6} the first and second moments of quadratic Hecke $L$-functions in $\mq(i)$ and $\mq(\omega)$, where $\omega=\frac {-1+\sqrt{3}i}{2}$. In \cite{Luo}, W. Luo considered the first two moments of cubic Hecke $L$-functions in $\mq(\omega)$.  Moments of various families of higher order Hecke $L$-functions were studied in \cites{FaHL, FHL,Diac, BFH}, using the method of double Dirichlet series. \newline

Although many results are available on moments of quadratic Dirichlet $L$-functions, less is known for the moments of higher order Dirichlet $L$-functions. For the family of cubic Dirichlet $L$-functions, this is investigated by S. Baier and M. P. Young in \cite{B&Y} and they also mentioned that it is plausible to extend their methods to study moments of quartic and sextic Dirichlet $L$-functions. Motivated by this, it is our goal in this paper to investigate the quartic case.  Writing $\chi_0$ for the principal character and letting $w: (0,\infty) \rightarrow \mr$ be any smooth, compactly supported function whose Fourier transform is denoted by $\widehat{w}$, we begin with a result that establishes an asymptotic expression for the first moment. \newline
\begin{theorem}
\label{firstmoment}
  With notations as above and assuming the truth of the Lindel\"of hypothesis, we have
\begin{equation*}
 \sum_{(q,2)=1}\;  \sumstar_{\substack{\chi \bmod{q} \\ \chi^4 = \chi_0}} L(\half, \chi) w\leg{q}{Q} = C Q \widehat{w}(0) + O(Q^{9/10 + \varepsilon}),
\end{equation*}
where $C$ is a positive explicit constant given in \eqref{eq:c} and the asterisk on the sum over $\chi$ restricts the sum to primitive characters $\chi$ such that $\chi^2$ remains primitive.
\end{theorem}

It is a conjecture that goes back to S. Chowla \cite{chow} that $L(1/2,\chi) \neq 0$ for all primitive Dirichlet characters $\chi$.  As a consequence of Theorem \ref{firstmoment}, a partial answer can be given for this conjecture for the character under our consideration.  We argue similar to the proof of \cite[Corollary 1.2]{B&Y}, using H\"older's inequality together with the well-known bound for the eighth moment of Dirichlet $L$-functions and the lower bound implied by the asymptotic formula in Theorem~\ref{firstmoment}, to obtain following non-vanishing result on the central values of quartic Dirichlet $L$-functions.
\begin{corollary} \label{coro:nonvanish}
Assume that the Lindel\"of hypothesis is true.  There exist infinitely many primitive Dirichlet characters $\chi$ of order $4$ such that $L(1/2, \chi) \neq 0$.  More precisely, the number of such characters with conductor $\leq Q$
is $\gg Q^{6/7- \varepsilon}$.
\end{corollary}

    For $m \in \mz, n \in \mz[i], (n,2)=1$, let $\psi_{m}((n)) = \leg{m}{n}_4$ where $\leg {\cdot}{\cdot}_4$ is the quartic residue symbol in $\mq(i)$ defined in Section \ref{sec2.4}. It is also shown there that $\psi_m$ is a quartic Hecke character of trivial infinite type modulo $16m$. Our next result concerns bounds on the second moment.
\begin{theorem}
\label{secmom}
  We have for any $Q\ge 1$ that
\begin{equation} \label{secondm}
\sum\limits_{\substack{q\le Q}}\ \sumstar\limits_{\substack{\chi \bmod q\\ \chi^4=\chi_0}} \left| L(\half+it,\chi) \right|^2 \ll Q^{7/6+\varepsilon} (1+|t|)^{1/2+\varepsilon}.
\end{equation}
 Further denote for any rational integer $m$,
\begin{equation}
\label{eq:HeckeL}
 L(s, \psi_m) =\sum_{\substack{n \in \mz[i] \\ n \equiv 1 \bmod{(1+i)^3}}} \psi_m(n) N(n)^{-s}.
\end{equation}
 Then we have
\begin{equation}
\label{eq:secondmomentHecke}
 \sumprime_{m \leq M} \left| L(1/2 + it, \psi_m) \right|^2 \ll M^{3/2 + \varepsilon} (1+|t|)^{5/4+ \varepsilon},
\end{equation}
where $\sum'$ means that the sum runs over square-free elements of $\mz$.
\end{theorem}

As the proof of the above result is similar to that of \cite[Theorem 1.3]{B&Y} (using the approximate functional equation and then the appropriate version of the large sieve inequality Lemmas~\ref{quartls} and~\ref{quarticlargesieve}), we shall therefore omit it in the paper. \newline

We point out here that the proof of \eqref{secondm} uses a large sieve inequality for quartic Dirichlet characters given in Lemma \ref{quarticlargesieve} by using the term $Q^{7/6}+Q^{2/3}M$ in the minimum there. Similarly, the proof of \eqref{eq:secondmomentHecke} follows from recalling the definition of $C_1(M,Q)$ in \cite[p. 907]{G&Zhao} and the bound $C_1(M,Q) \ll (QM)^{\varepsilon}(Q^{5/4}+Q^{1/2}M)$ by Lemma \ref{quarticlargesieve}. A variant of this for the cubic case plays a key role in the treatment of the error term in the work of Baier and Young \cite{B&Y} on the first moment of cubic Dirichlet $L$-functions.  However, our Lemma \ref{quarticlargesieve} is just short of producing admissible error term in our first moment result, even after incorporating the recursive technique of Baier and Young (see third arXiv version of the paper \cite{B&Y} arXiv:0804.2233v3).  This observation drives us to seek for bounds of individual quartic Dirichlet $L$-functions in order to control the error term and leads us to derive our Theorem \ref{firstmoment} under the Lindel\"of hypothesis since the current best known subconvexity bounds for the $L$-functions under our consideration do not provide a satisfactory error term either. \newline

\subsection{Notations} The following notations and conventions are used throughout the paper.\\
\noindent $e(z) = \exp (2 \pi i z) = e^{2 \pi i z}$. \newline
$f =O(g)$ or $f \ll g$ means $|f| \leq cg$ for some unspecified
positive constant $c$. \newline
$\mu_{[i]}$ denotes the M\"obius function on $\mz[i]$. \newline
$\zeta_{\mq(i})(s)$ is the Dedekind zeta function for the field $\mq(i)$.

\section{Preliminaries}
\label{sec 2}

In this section we provide various tools used throughout the paper.
\subsection{Quartic symbol and primitive quartic Dirichlet characters}
\label{sec2.4}
   The symbol $(\leg{\cdot}{n})_4$ is the quartic
residue symbol in the ring $\mz[i]$.  For a prime $\varpi \in \mz[i]$
with $N(\varpi) \neq 2$, the quartic residue symbol is defined for $a \in
\mz[i]$, $(a, \varpi)=1$ by $\leg{a}{\varpi}_4 \equiv
a^{(N(\varpi)-1)/4} \pmod{\varpi}$, with $\leg{a}{\varpi}_4 \in \{
\pm 1, \pm i \}$. When $\varpi | a$, we define
$\leg{a}{\varpi}_4 =0$.  Then the quartic residue symbol can be extended
to any composite $n$ with $(N(n), 2)=1$ multiplicatively. We further set $\leg {\cdot}{n}_4=1$ when $n$ is a unit in $\mz[i]$. We also define the quadratic residue symbol $\leg {\cdot}{n}$ so that $\leg {\cdot}{n}=\leg {\cdot}{n}^2_4$. \newline

 Note that in $\mz[i]$, every ideal co-prime to $2$ has a unique
generator congruent to 1 modulo $(1+i)^3$.  Such a generator is
called primary. Observe that
$n=a+bi, a, b \in \mz $ in $\mz[i]$ is congruent to $1
\bmod{(1+i)^3}$ if and only if $a \equiv 1 \pmod{4}, b \equiv
0 \pmod{4}$ or $a \equiv 3 \pmod{4}, b \equiv 2 \pmod{4}$ by \cite[Lemma
6, p. 121]{I&R}. It follows from this that we have
\begin{align}
\label{a&b}
  a \equiv (-1)^{\frac {N(n)-1}{4}}  \quad \pmod 4,  \quad b \equiv 1-(-1)^{\frac {N(n)-1}{4}}  \quad \pmod 4.
\end{align}

Recall that \cite[Theorem 6.9]{Lemmermeyer} the quartic reciprocity law states
that for two primary integers  $m, n \in \mz[i]$,
\begin{align*}
 \leg{m}{n}_4 = \leg{n}{m}_4(-1)^{((N(n)-1)/4)((N(m)-1)/4)}.
\end{align*}
  Also, the supplement laws to the quartic reciprocity law states that for $n=a+bi$ being primary,
\begin{align}
\label{2.05}
  \leg {i}{n}_4=i^{(1-a)/2} \qquad \mbox{and} \qquad  \hspace{0.1in} \leg {1+i}{n}_4=i^{(a-b-1-b^2)/4}.
\end{align}

    This implies that for any $n \equiv 1 \pmod {16}$, we have
\begin{align*}
  \ \leg {i}{n}_4=\leg {1+i}{n}_4=1.
\end{align*}
Hence the functions $n \rightarrow \leg {i}{n}_4$ and $n \rightarrow \leg {1+i}{n}_4$ can be regarded as Hecke characters $\pmod {16}$ of trivial infinite type. From this we see that $\psi_m$ is a quartic Hecke character of trivial infinite type modulo $16m$ for every $m \in \mz$. \newline

  Using the quartic residue symbols, we have the following classification of all the primitive quartic Dirichlet characters of conductor $q$ co-prime to $2$:
\begin{lemma}
\label{lemma:quarticclass}
 The primitive quartic Dirichlet characters of conductor $q$ coprime to $2$ such that their squares remain primitive are of the form $\chi_n:m \mapsto \leg{m}{n}_4$ for some $n \in \mz[i]$, $n \equiv 1 \pmod{(1+i)^3}$, $n$ square-free and not divisible by any rational primes, with norm $N(n) = q$.
\end{lemma}

The above result is given in \cite[Section 2.2]{G&Zhao}, but the correspondence given there is not exact.  We take this opportunity to thank Francesca Balestrieri and Nick Rome for some helpful discussions on this topic.

\subsection{The Gauss sums}
     For any $n \in  \mz[i]$, $(n,2)=1$, the quartic Gauss sum $g(n)$ is defined by
\begin{align*}
   g(n) = \sum_{x \bmod{n}} \leg{x}{n}_4 \widetilde{e}\leg{x}{n},
\end{align*}
    where $\widetilde{e}(z) =\exp \left( 2\pi i  (\frac {z}{2i} - \frac {\overline{z}}{2i}) \right)$. \newline

   The following well-known formula (see \cite[p. 195]{P}) holds for all $n$:
\begin{align}
\label{2.1}
   |g(n)|& =\begin{cases}
    \sqrt{N(n)} \qquad & \text{if $n$ is square-free}, \\
     0 \qquad & \text{otherwise}.
    \end{cases}
\end{align}

    More generally, we define for $n,k \in \mz[i], (n,2)=1$,
\begin{align*}
 g(k,n) = \sum_{x \bmod{n}} \leg{x}{n}_4 \widetilde{e}\leg{kx}{n}.
\end{align*}

   We note two properties of $g(r,n)$ that can be found in \cite{Diac}:
\begin{align}
\label{eq:gmult}
 g(rs,n) & = \overline{\leg{s}{n}_4} g(r,n), \quad (s,n)=1, \\
\label{2.03}
   g(r,n_1 n_2) &=\leg{n_2}{n_1}_4\leg{n_1}{n_2}_4g(r, n_1) g(r, n_2), \quad (n_1, n_2) = 1.
\end{align}

   We further define the Gauss sum $\tau(\chi)$ associated to a primitive Dirichlet character $\chi \pmod q$ by
\begin{equation*}
  \tau(\chi) =  \sum_{1 \leq x \leq q}\chi(x) e \left( \frac{x}{q} \right).
\end{equation*}

   When $\chi$ is a primitive quartic Dirichlet character identified with $\chi_n$ by Lemma \ref{lemma:quarticclass}, we
have (\cite[p. 894]{G&Zhao})
\begin{equation}
\label{tau}
 \tau(\chi_n) =   \leg {\overline{n}}{n}_4g(n).
\end{equation}

  Here we have $n\equiv 1 \pmod {(1+i)^3}$ and $n$ free of rational prime divisors. If we write $n=a+bi$ with $a, b \in \mz$, then we deduce from this that $(a, b)=1$ so that
\begin{align} \label{nbar/n}
  \leg {\overline{n}}{n}_4= \leg {a-bi}{a+bi}_4=\leg {2a}{a+bi}_4=\leg {2(-1)^{\frac {N(n)-1}{4}}}{a+bi}_4\leg {(-1)^{\frac {N(n)-1}{4}}a}{a+bi}_4.
\end{align}

   As $(-1)^{\frac {N(n)-1}{4}}a$ is primary according to \eqref{a&b}, we have by the quartic reciprocity law,
\begin{align} \label{nbar/n1}
   \leg {(-1)^{\frac {N(n)-1}{4}}a}{a+bi}_4=(-1)^{((N(a)-1)/4)((N(n)-1)/4)} \leg {a+bi}{a}_4= \leg {a+bi}{a}_4= \leg {bi}{a}_4= \leg {b}{a}_4 \leg {i}{a}_4= \leg {i}{a}_4,
\end{align}
   where the last equality follows from \cite[Proposition 9.8.5]{I&R}, which states that for $a, b \in \mz, (a, 2b)=1$ ,
\[  \leg {b}{a}_4=1. \]

   It follows from \eqref{2.05} that
\begin{align} \label{nbar/n2}
   \leg {i}{a}_4=i^{(1-(-1)^{\frac {N(n)-1}{4}}a)/2}=(-1)^{\frac {a^2-1}{8}}.
\end{align}

On the other hand, note that by the definition that
\begin{align} \label{nbar/n3}
   \leg {(-1)^{\frac {N(n)-1}{4}}}{a+bi}_4=(-1)^{\frac {N(n)-1}{4}\cdot \frac {N(n)-1}{4}}=(-1)^{\frac {N(n)-1}{4}}=\leg {-1}{n}_4.
\end{align}

   We then deduce that, putting together \eqref{nbar/n}, \eqref{nbar/n1}, \eqref{nbar/n2} and \eqref{nbar/n3},
\begin{align*}
  \leg {\overline{n}}{n}_4= \leg {-2}{n}_4(-1)^{\frac {a^2-1}{8}}=\begin{cases}
     \leg {-2i}{n}_4=\overline{ \leg {(-2i)^3}{n}}_4 \qquad & \text{if $\leg {-1}{n}_4=1$}, \\ \\
      i^{-1}\leg {-2i}{n}_4=i^{-1}\overline{\leg{(-2i)^3}{n}}_4 \qquad & \text{if $\leg {-1}{n}_4=-1$ }.
    \end{cases}
\end{align*}

Thus from this and \eqref{tau},
\begin{align}
\label{taug}
  \tau(\chi_n)=\begin{cases}
    \overline{ \leg {(-2i)^3}{n}}_4 g(n) \qquad & \text{if $\leg {-1}{n}_4=1$}, \\ \\
     i^{-1}\overline{\leg{(-2i)^3}{n}}_4 g(n) \qquad & \text{if $\leg {-1}{n}_4=-1$ }.
    \end{cases}
\end{align}

\subsection{The approximate functional equation}
    Let $G(s)$ be any even function which is holomorphic and bounded in the strip $-4<\Re(s)<4$ satisfying $G(0)=1$. It follows from \cite[Theorem 5.3]{iwakow} that we have the following approximate functional equation for Dirichlet $L$-functions.
\begin{prop}
\label{prop:AFE}
Let $\chi$ be a primitive Dirichlet character of conductor $q$. For any $\alpha \in \mc, j \in \{ \pm 1 \}$, let
\begin{equation}
\label{2.8}
 a_j=\frac {1-j}2, \quad  \epsilon(\chi) = i^{-a_{\chi(-1)}} q^{-1/2} \tau(\chi), \quad  X_{\alpha,j} = \left(\frac{q}{\pi}\right)^{-\alpha} \frac{\Gamma\left(\tfrac{1/2 + a_j- \alpha}{2}\right)}{\Gamma\left(\tfrac{1/2 + a_j + \alpha}{2}\right)}.
\end{equation}
  We define
\begin{equation*}
  V_{\alpha, j}(x) = \frac{1}{2\pi i} \int\limits_{(2)} \frac{G(s)}{s} \gamma_{\alpha, j}(s) x^{-s} \dif s, \quad \text{where} \quad \gamma_{\alpha, j}(s) = \pi^{-s/2} \frac{\Gamma\left(\tfrac{1/2 + a_j+\alpha+ s}{2}\right)}{\Gamma\left(\tfrac{1/2 + a_j + \alpha}{2}\right)}.
\end{equation*}
  Furthermore, let $A$ and $B$ be positive real numbers such that $AB = q$.  Then for any $|\Re(\alpha)| < 1/2$ we have
\begin{equation*}
L(1/2 + \alpha, \chi) = \sum_{m=1}^{\infty} \frac{\chi(m)}{m^{1/2 + \alpha}} V_{\alpha, \chi(-1)}\left(\frac{m}{A}\right) + \epsilon(\chi) X_{\alpha, \chi(-1)} \sum_{m=1}^{\infty} \frac{\overline{\chi}(m)}{m^{1/2 - \alpha}} V_{-\alpha, \chi(-1)}\left(\frac{m}{B}\right).
\end{equation*}
\end{prop}
\noindent For $\alpha = 0$ we set $\gamma_1=\gamma_{0,1}$, $\gamma_{-1}=\gamma_{0,-1}$, $V_1=V_{0,1}$, $V_{-1}=V_{0,-1}$. \newline

With a suitable $G(s)$ (for example $G(s)=e^{s^2}$), we have for any $c>0$ (see \cite[Proposition 5.4]{HIEK}):
\begin{align*}
  V_{\alpha,j} \left(\xi \right) \ll \left( 1+\frac{\xi}{1+|\alpha|} \right)^{-c}.
\end{align*}

   On the other hand, when $G(s)=1$, we have (see \cite[Lemma 2.1]{sound1}) that $V_{\pm 1}(\xi)$ is real-valued and smooth on $[0, \infty)$ and for the $j$-th derivative of $V_{\pm 1}(\xi)$,
\begin{equation} \label{2.07}
      V_{\pm 1}\left (\xi \right) = 1+O(\xi^{1/2-\epsilon}) \; \mbox{for} \; 0<\xi<1   \quad \mbox{and} \quad V^{(j)}_{\pm 1}\left (\xi \right) =O(e^{-\xi}) \; \mbox{for} \; \xi >0,\; j \geq 0.
\end{equation}

\subsection{Analytic behavior of Dirichlet series associated with Gauss sums}
\label{section: smooth Gauss}
  For any Hecke character $\chi \pmod {16}$ of trivial infinite type, we let
\begin{align*}
   h(r,s;\chi)=\sum_{\substack{(n,r)=1 \\ n \equiv 1 \bmod{(1+i)^3}}}\frac {\chi(n)g(r,n)}{N(n)^s}.
\end{align*}

     The following lemma gives the analytic behavior of $h(r,s;\chi)$ on $\Re(s) >1$.
\begin{lemma}{\cite[Lemma 2.5]{G&Zhao1}}
\label{lem1} The function $h(r,s;\chi)$ has meromorphic continuation to the entire complex plane. It is holomorphic in the
region $\sigma=\Re(s) > 1$ except possibly for a pole at $s = 5/4$. For any $\varepsilon>0$, letting $\sigma_1 = 3/2+\varepsilon$, then for $\sigma_1 \geq \sigma \geq \sigma_1-1/2$, $|s-5/4|>1/8$, we have for $t=\Im(s)$,
\[ h(r,s;\chi) \ll N(r)^{(\sigma_1-\sigma)/2+\varepsilon}(1+t^2)^{3(\sigma_1-\sigma)/2+\varepsilon}. \]
 Moreover, the residue satisfies
\[ \mathop{\mathrm{Res}}_{s=5/4}h(r,s;\chi) \ll N(r)^{1/8+\varepsilon}. \]
\end{lemma}

\subsection{The large sieve with quartic Dirichlet characters}
The following large sieve inequality for quadratic residue symbols is needed in the proof.
\begin{lemma}{\cite[Theorem 1]{Onodera}} \label{quartls}
Let $M,N$ be positive integers, and let $(a_n)_{n\in \mathbb{N}}$ be an arbitrary sequence of complex numbers, where $n$ runs over $\mz[i]$. Then we have for any $\varepsilon > 0$,
\begin{equation*}
 \sumprime_{\substack{m \in \mz[i] \\ (m,2)=1 \\ N(m) \leq M}} \left| \ \sumprime_{\substack{n \in \mz[i] \\ (n,2)=1 \\N(n) \leq N}} a_n \leg{n}{m} \right|^2
 \ll_{\varepsilon} (MN)^{\varepsilon}(M + N) \sum_{N(n) \leq N} |a_n|^2,
\end{equation*}
  where $\sum'$ means that the sum runs over square-free elements of $\mz[i]$ and $(\frac
{\cdot}{m})$ is the quadratic residue symbol.
\end{lemma}

    Besides the notation $C_1(M,Q)$ mentioned in the Introduction, we also recall the norm $C_2(M,Q)$ defined in \cite[p. 907]{G&Zhao}. Notice that for square-free $n \in \mz[i]$ satisfying $(n,2)=1$, the square of the quartic symbol $\leg {\cdot }{n}_4^2$ becomes the quadratic residue symbol $\leg {\cdot }{n}$. Thus, using Lemma \ref{quartls} in \cite[(39)]{G&Zhao} and proceed to the estimate \cite[(40)]{G&Zhao}, we see that the estimation given in \cite[(24)]{G&Zhao} can be replaced by
\begin{equation}
\label{C2e1}
  C_2(M,Q) \ll (QM)^{\varepsilon}\left(M +Q^{3/2} \right),
\end{equation}
  As we have by \cite[(23)]{G&Zhao} that $C_1(M, Q) \leq C_2(M,Q)$, we see that the bound given in \eqref{C2e1} is also valid for $C_1(M,Q)$, which leads to the following large sieve inequality for quartic Dirichlet characters.
\begin{lemma} \label{quarticlargesieve}
Let $(a_m)_{m\in \mathbb{N}}$ be an
arbitrary sequence of complex numbers. Then
\begin{equation*}
\begin{split}
 \sum\limits_{\substack{Q<q\le 2Q}} \ \sumstar\limits_{\substack{\chi \bmod q\\ \chi^4=\chi_0}} &
\left| \ \sumstar\limits_{\substack{M<m\le 2M}} a_m
\chi(m)\right|^2\\
\ll &
(QM)^{\varepsilon}\min\left\{Q^{3/2}+M,Q^{5/4}+Q^{1/2}M,Q^{7/6}+Q^{2/3}M,
Q+Q^{1/3}M^{5/3}+M^{7/3}\right\}
\sumstar\limits_{\substack{M<m\le 2M}}\left| a_m \right|^2,
\end{split}
\end{equation*}
where the asterisk on the sum over $\chi$ restricts the sum to
primitive characters whose square remain primitive and the asterisks attached to the sum over $m$
indicates that $m$ runs over square-free integers.
\end{lemma}

\section{Proof of Theorem \ref{firstmoment}}
  We start by setting
\begin{equation*}
\mathcal{M} := \sum_{(q,2)=1}\;  \sumstar_{\substack{\chi \bmod{q} \\ \chi^4 = \chi_0}} L(\half, \chi) w\leg{q}{Q}.
\end{equation*}

Applying Proposition~\ref{prop:AFE} (the approximate functional equation) with $G(s)=1, A_q B = q$ yields $\mathcal{M} = \mathcal{M}_1 + \mathcal{M}_2$, where
\begin{align*}
 \mathcal{M}_1 &= \sum_{(q,2)=1}\;  \sumstar_{\substack{\chi \bmod{q} \\ \chi^4 = \chi_0}}\ \sum_{m=1}^{\infty} \frac{\chi(m)}{\sqrt{m}} V_{\chi(-1)}\leg{m}{A_q} w\left(\frac{q}{Q}\right), \\
 \mathcal{M}_2 &= \sum_{(q,2)=1}\;  \sumstar_{\substack{\chi \bmod{q} \\ \chi^4 = \chi_0}} \epsilon(\chi)\sum_{m=1}^{\infty} \frac{\overline{\chi}(m)}{\sqrt{m}} V_{\chi(-1)}\leg{m}{B} w\left(\frac{q}{Q}\right).
\end{align*}

For each primitive quartic Dirichlet character $\chi$  whose square remains primitive, we write it as $\chi_n$ via the correspondence given in Lemma \ref{lemma:quarticclass}. It then follows from \eqref{tau}, \eqref{2.8} and this correspondence that
\begin{align*}
 \epsilon(\chi_n) = \epsilon(\chi) = i^{-a_{\chi_n(-1)}} N(n)^{-1/2}\left(\frac{\bar{n}}{n}\right)_4g(n).
\end{align*}

The above allows us to further decompose $M_1$ and $M_2$ as $\mathcal{M}_1 =\mathcal{M}^+_1 +\mathcal{M}^-_1$ and $\mathcal{M}_2 =\mathcal{M}^+_2 +\mathcal{M}^-_2$, where
\begin{align*}
 \mathcal{M}^{\pm}_1 &= \sumprimeprime_{\substack {n \equiv 1 \bmod (1+i)^3}}\frac {1 \pm \chi_n(-1)}{2}\ \sum_{m=1}^{\infty} \frac{\chi_n(m)}{\sqrt{m}}  V_{\chi_n(-1)}\leg{m}{A_n} w\left(\frac{N(n)}{Q}\right), \\
 \mathcal{M}^{\pm}_2 &= \sumprimeprime_{\substack {n \equiv 1 \bmod (1+i)^3}} \frac {1 \pm \chi_n(-1)}{2} \epsilon(\chi_n) \sum_{m=1}^{\infty} \frac{\overline{\chi}_n(m)}{\sqrt{m}}  V_{\chi_n(-1)}\leg{m}{B} w\left(\frac{N(n)}{Q}\right).
\end{align*}
Here $\sum''$ means that the sum runs over square-free elements $n \in \mathbb{Z}[i]$ that have no rational prime divisors.
  Also, the parameters $A_n, B$ sastify that $A_nB=N(n)$. We now introduce another parameter $A$ defined by $AB=Q$ and we note that we have $A_n = A N(n)/Q \asymp A$ for all $n$ under consideration because of the compact support of $w$. \newline

It remains to evaluate $\mathcal{M}^{\pm}_1$ and $\mathcal{M}^{\pm}_2$.  As the arguments are similar, we will only evaluate $\mathcal{M}^{+}_1$ and $\mathcal{M}^{+}_2$ in what follows. We summarize the results in the following lemma.
\begin{lemma}
\label{lem2}
 We have
\begin{align}
\label{eq:M1estimate}
 \mathcal{M}^{\pm}_1 =& \frac 12 C Q \widetilde{w}(1) + O \left( Q^{1/2 + \varepsilon} A^{1/2+ \varepsilon} + Q A^{-1/4 + \varepsilon} \right), \\
 \label{eq:M2bound}
 \mathcal{M}^{\pm}_2 \ll & Q^{3/4 + \varepsilon} B^{3/4 + \varepsilon},
\end{align}
where the constant $C$ is given explicitly in \eqref{eq:c}.
\end{lemma}
Theorem \ref{firstmoment} follows from Lemma~\ref{lem2} by setting $B = Q^{1/5}$ and $A = Q^{4/5}$ in Lemma  \ref{lem2}.  Thus the remainder of the paper is devoted to the proof of this lemma.

\subsection{Evaluating $\mathcal{M}^+_1$, the main term}
\label{section:M1}
  We write $\mu$ for M\"{o}bius function and we define $\mu_{\mz}(d) = \mu(|d|)$ for $d \in \mz$. Then for any $n\equiv 1 \pmod{(1+i)^3}$, we note the following relation
\begin{equation}
\label{eq:ratmob}
 \sum_{\substack{d |n, d \in \mz \\ d \equiv 1 \bmod 4}} \mu_{\mz}(d) =
\begin{cases}
 1, \quad \text{$n$ has no rational prime divisors}, \\
 0, \quad \text{otherwise}.
\end{cases}
\end{equation}
  We apply the above and change variables $n \rightarrow dn$ to the sum over $n$.  Note that any square-free $d \in \mz, d \equiv 1 \pmod 4$ is also square-free when regarded as an element of $\mz[i]$ and this implies that the condition that $dn$ is square-free then is equivalent to $n$ being square-free and $(d,n) = 1$.  We then deduce that $M^+_1=M^+_{1,1}+M^+_{1,2}$, with
\begin{align*}
 \mathcal{M}^+_{1,1} &= \frac 12\sum_{\substack{d \in \mz \\ d \equiv 1 \bmod 4 }} \mu_{\mz}(d) \sum_{m=1}^{\infty} \frac{\leg{m}{d}_4}{\sqrt{m}} \sumprime_{\substack{n \equiv 1 \bmod (1+i)^3 \\ (n,d) = 1}} \leg{m}{n}_4 V_1\left(\frac{m}{A} \frac{Q}{N(nd)} \right) w\left(\frac{N(nd)}{Q}\right), \\
 \mathcal{M}^+_{1,2} &= \frac 12\sum_{\substack{d \in \mz \\ d \equiv 1 \bmod 4 }} \mu_{\mz}(d) \sum_{m=1}^{\infty} \frac{\leg{-m}{d}_4}{\sqrt{m}} \sumprime_{\substack{n \equiv 1 \bmod (1+i)^3 \\ (n,d) = 1}} \leg{-m}{n}_4 V_1\left(\frac{m}{A} \frac{Q}{N(nd)} \right) w\left(\frac{N(nd)}{Q}\right),
\end{align*}
  where, as before, $\sum'$ denotes a sum over the square-free elements of $\mz[i]$. \newline

  We first evaluate $\mathcal{M}^+_{1,1}$ by using M\"{o}bius inversion to detect the square-free condition that $n$.  So
\begin{equation*}
 \mathcal{M}^+_{1,1} = \frac 12 \sum_{\substack{d \in \mz \\ d \equiv 1 \bmod 4 }} \mu_{\mz}(d) \sum_{\substack{ l \equiv 1 \bmod (1+i)^3 \\ (l,d) = 1}} \mu_{[i]}(l) \sum_{m=1}^{\infty} \frac{\leg{m}{d l^2}_4}{\sqrt{m}} \mathcal{M}_1(d,l,m),
\end{equation*}
where
\begin{equation*}
 \mathcal{M}_1(d,l,m) = \sum_{\substack{n \equiv 1 \bmod (1+i)^3 \\ (n,d) = 1}} \leg{m}{n}_4 V_1\left(\frac{m}{A} \frac{Q}{N(ndl^2)} \right) w\left(\frac{N(ndl^2)}{Q}\right).
\end{equation*}
Now Mellin inversion gives
\begin{equation*}
 V_1\left(\frac{m}{A} \frac{Q}{N(ndl^2)} \right) w\left(\frac{N(ndl^2)}{Q}\right) = \frac{1}{2 \pi i} \int\limits_{(2)} \leg{Q}{N(ndl^2)}^s \widetilde{f}(s) \dif s,
\end{equation*}
where
\begin{equation*}
\widetilde{f}(s) = \int\limits_0^{\infty} V_1\left(\frac{m}{A} x^{-1} \right) w(x) x^{s-1} \dif x.
\end{equation*}
Integration by parts and using \eqref{2.07} shows that $\widetilde{f}(s)$ is a function satisfying the bound
\begin{equation*}
 \widetilde{f}(s) \ll (1 + |s|)^{-E} \left( 1 + m/A \right)^{-E},
\end{equation*}
for any $\Re(s) >0$ and any integer $E>0$. \newline

Thus we deduce from the above bound that
\begin{equation*}
 \mathcal{M}_1(d,l,m) = \frac{1}{2 \pi i} \int\limits_{(2)} \leg{Q}{N(dl^2)}^s L(s, \psi_{md^4}) \widetilde{f}(s) \dif s,
\end{equation*}
   with $L(s,\psi_{md^4})$ given in \eqref{eq:HeckeL}. \newline

We estimate $\mathcal{M}^+_{1,1}$  by moving the contour to the line $\Re s = 1/2$.  Observe that the Hecke $L$-function has a pole at $s=1$ when $m$ is a fourth power.  We write $\mathcal{M}_0$ for the contribution to $\mathcal{M}^+_{1,1}$ from these residues and $\mathcal{M}_1'$ for the remainder. \newline

  We evaluate $\mathcal{M}_0$ by noting that
\begin{equation*}
 \mathcal{M}_0 = \frac 12 \sum_{\substack{d \in \mz \\ d \equiv 1 \bmod{4} }} \mu_{\mz}(d) \sum_{\substack{ l \equiv 1 \bmod (1+i)^3  \\ (l,d) = 1}} \mu_{[i]}(l) \sum_{m=1}^{\infty} \frac{\leg{m}{d l^2}_4}{\sqrt{m}} \frac{Q}{N(dl^2)} \widetilde{f}(1) \mathop{\text{Res}}_{s=1} L(s, \psi_{md^4}),
\end{equation*}
where using the Mellin convolution formula shows that
\begin{equation}
\label{w}
 \widetilde{f}(1) = \int\limits_0^{\infty} V_1\left(\frac{m}{A} x^{-1} \right) w(x) \dif x = \frac{1}{2 \pi i} \int\limits_{(2)} \leg{A}{m}^s \widetilde{w}(1+s) \frac{G(s)}{s} \gamma_1(s) \dif s,  \quad  \mbox{with} \; \widetilde{w}(s) = \int\limits_0^{\infty} w(x) x^{s-1} \dif x.
\end{equation}	
From the discussions in Section \ref{sec2.4}, we see that $\psi_{md^4}$ is the principal character only when $m$ is a fourth power, in which case
\begin{equation*}
 L(s, \psi_{md^4}) = \zeta_{\mq(i)}(s) \prod_{\varpi | 2dm} \left( 1 - N(\varpi)^{-s} \right),
\end{equation*}
   where $\varpi$ denotes a prime in $\mz[i]$ in this section. \newline

Let $C_0 = \frac{\pi}{4}$, the residue of $\zeta_{\mq(i)}(s)$ at $s=1$, so that
\begin{equation*}
 \mathcal{M}_0 = \frac 12 C_0 Q \sum_{m=1}^{\infty} \frac{\widetilde{f}(1)}{m^{2}} \prod_{\varpi | 2m} \left( 1- N(\varpi)^{-1} \right)  \sum_{\substack{d \in \mz, (d,m) =1\\ d \equiv 1 \bmod{4} }} \frac{\mu_{\mz}(d)}{d^2}\prod_{\varpi | d} \left( 1 - N(\varpi)^{-1} \right) \sum_{\substack{(l,md)=1 \\ l \equiv 1 \bmod{(1+i)^3}}} \frac{\mu_{[i]}(l)}{N(l^2)}.
\end{equation*}

   Computing the sum over $l$ explicitly, we obtain that
\begin{align*}
 \mathcal{M}_0 &=\frac 12 C_0 \zeta^{-1}_{\mq(i)}(2) Q \sum_{m=1}^{\infty} \frac{\widetilde{f}(1)}{m^{2}}\prod_{\varpi | 2m} \left( 1- N(\varpi)^{-1} \right) \sum_{\substack{d \in \mz, (d,m) =1\\ d \equiv 1 \bmod{4} }} \frac{\mu_{\mz}(d)}{d^2}\prod_{\varpi | d} \left( 1 - N(\varpi)^{-1} \right) \prod_{\varpi | 2md} \left( 1- N(\varpi)^{-2} \right)^{-1} \\
 &= \frac 12  C_0 \zeta^{-1}_{\mq(i)}(2) Q \sum_{m=1}^{\infty} \frac{\widetilde{f}(1)}{m^{2}}\prod_{\varpi | 2m} \left( 1+ N(\varpi)^{-1} \right)^{-1} \sum_{\substack{d \in \mz, (d,m) =1\\ d \equiv 1 \bmod{4} }} \frac{\mu_{\mz}(d)}{d^2}\prod_{\varpi | d} \left( 1 + N(\varpi)^{-1} \right)^{-1}.
\end{align*}

   We define
\begin{align*}
  C_1= \sum_{\substack{d \in \mz \\ d \equiv 1 \bmod{4} }} \frac{\mu_{\mz}(d)}{d^2}\prod_{\varpi | d} \left( 1 + N(\varpi)^{-1} \right)^{-1}.
\end{align*}

   It is clear that $C_1$ is a constant. Using this, we have that
\begin{align*}
 \mathcal{M}_0 = \frac 12 C_0C_1 \zeta^{-1}_{\mq(i)}(2) Q \sum_{m=1}^{\infty} \frac{\widetilde{f}(1)}{m^{2}} \prod_{\varpi | 2m} \left( 1+ N(\varpi)^{-1} \right)^{-1}   \prod_{p | m/(m,2)} \left( 1-p^{-2}\prod_{\varpi | p} (1 + N(\varpi)^{-1})^{-1} \right)^{-1}.
\end{align*}

Let
\begin{equation*}
 Z(u) = \sum_{m=1}^{\infty} m^{-u} \prod_{\varpi | 2m} \left( 1+ N(\varpi)^{-1} \right)^{-1}   \prod_{p | m/(m,2)} \left( 1-p^{-2}\prod_{\varpi | p} (1 + N(\varpi)^{-1})^{-1} \right)^{-1},
\end{equation*}
which is holomorphic and bounded for $\Re(u) \geq 1 + \delta > 1$.  Then
\begin{equation*}
 \mathcal{M}_0 = \frac 12 C_0C_1 \zeta^{-1}_{\mq(i)}(2)  Q \frac{1}{2 \pi i} \int\limits_{(1)} A^s Z(2 + 4s) \widetilde{w}(1+s) \frac{G(s)}{s} \gamma_1(s) \dif s.
\end{equation*}
We move the contour of integration to $-1/4 + \varepsilon$, crossing a pole at $s=0$ only.  The integral over the new contour is $O(A^{-1/4 + \varepsilon} Q)$, while residue of the pole at $s=0$ gives
\begin{equation}
\label{eq:c}
\frac 12C Q \widetilde{w}(1), \quad \text{where} \quad C= C_0C_1  \zeta^{-1}_{\mq(i)}(2) Z(2).
\end{equation}
Note that $Z(u)$ converges absolutely at $u=2$ so that one may express $Z(2)$ explicitly as an Euler product if interested.
We then conclude that
\begin{align}
\label{m0}
 \mathcal{M}_0 = \frac 12C Q \widetilde{w}(1)+O \left( Q A^{-1/4 + \varepsilon} \right).
\end{align}

\subsection{Evaluating $\mathcal{M}^+_1$, the remainder term}
\label{section:remainderterm}

 In this section, we estimate $\mathcal{M}'_1$ and $\mathcal{M}^+_{1,2}$. Since the arguments are similar, we shall only estimate $\mathcal{M}'_1$  here. We bound everything with absolute values to see that for some large $E \in \natn$,
\begin{equation*}
|\mathcal{M}_1'| \ll  \sum_{d \leq c_1 \sqrt{Q}} \sum_{N(l) \leq c_2 \sqrt{Q}} \frac{1}{\sqrt{N(dl^2)}} \sum_{m} \frac{\sqrt{Q}}{\sqrt{m}} (1+m/A)^{-E} \int\limits^{\infty}_0 \left| L(1/2 + it, \psi_m) \right| \left( 1+|t| \right)^{-E} \dif t.
\end{equation*}
Here $c_1$ and $c_2$ are constants, chosen according to the size of the support the weight function $w$.  In view of the factor $(1+m/A)^{-E}$, we may truncate the sum over $m$ above to $m \leq M \ll A^{1+\varepsilon}$ for $\varepsilon>0$ with a small error.  \newline

Assuming the truth of the Lindel\"of hypothesis, we have for any $\varepsilon>0$ (see \cite[Corollary 5.20]{iwakow}),
\begin{align*}
 |L(1/2 + it, \psi_m)| \ll (|m|(1+|t|))^{\epsilon}.
\end{align*}

  We apply this to bound the sum over $m$ as
\begin{equation*}
 \sum_{m \leq M} \frac{1}{\sqrt{m}} |L(1/2 + it, \psi_m)| \ll M^{1/2 + \varepsilon} (1+|t|)^{\varepsilon}.
\end{equation*}

   Since $d \in \mz$, we have $N(d) = d^2$ so that summing trivially over $d$ and $l$, we obtain
\begin{equation*}
|\mathcal{M}_1'|+|\mathcal{M}^+_{1,2}| \ll Q^{1/2 + \varepsilon} A^{1/2+ \varepsilon}.
\end{equation*}
  This combined with \eqref{m0} gives \eqref{eq:M1estimate}.

\subsection{Estimating $\mathcal{M}^+_2$}
\label{section:M2}

Using \eqref{taug}, we have $M^+_2=M^+_{2,1}+M^+_{2,2}$, where
\begin{align*}
 \mathcal{M}^+_{2,1} &= \frac {1}{2} \sum_{m=1}^{\infty} \frac{1}{\sqrt{m}} V_1\left(\frac{m}{B}\right) \sumprime_{\substack{ n\equiv 1 \bmod (1+i)^3 }} \frac{\overline{\chi}_{n}(8im) g(n)}{\sqrt{N(n)}} w\left(\frac{N(n)}{Q}\right), \\
  \mathcal{M}^+_{2,2} &= \frac {1}{2} \sum_{m=1}^{\infty} \frac{1}{\sqrt{m}} V_1\left(\frac{m}{B}\right) \sumprime_{\substack{ n\equiv 1 \bmod (1+i)^3 }} \frac{\overline{\chi}_{n}(-8im) g(n)}{\sqrt{N(n)}} w\left(\frac{N(n)}{Q}\right).
\end{align*}
To estimate the above expressions, we need the following result.
\begin{lemma}
\label{lemma:sumofGauss}
 For any $l \in \mz[i]$, we have
\begin{equation*}
\label{eq:sumofGausssums}
H'(l,Q):=\sumprime_{\substack{n \in \mz[i] \\ n\equiv 1 \bmod (1+i)^3}} \frac{\overline{\chi}_n(l) g(n) }{\sqrt{N(n)}} w\left(\frac{N(n)}{Q}\right) \ll  Q^{3/4 + \varepsilon} N(l)^{1/8+\varepsilon}.
\end{equation*}
\end{lemma}

To prove Lemma \ref{lemma:sumofGauss}, we first use
\eqref{eq:ratmob} to remove the condition that $n$ has no rational prime divisors. Note that it follows from  \eqref{2.03} and the quartic reciprocity that for $d \in \mz, n \in \mz[i], (dn,2)=1, (d,n)=1$ and $d$ primary,
\begin{align*}
 g(dn) = \leg{n}{d}_4\leg{d}{n}_4 g(d) g(n)= \overline{\chi}_n(d^2) g(d) g(n).
\end{align*}

  We use the notation $\widetilde{g}(d) = g(d) N(d)^{-1/2}$ so that $|\widetilde{g}(d)| \leq 1$ by \eqref{2.1}. It follows further from \eqref{2.1} that $g(n) = 0$ unless $n$ is square-free. This gives that
\begin{equation*}
\label{eq:M2sieved}
 H'(l,Q) = \sum_{\substack{d \in \mz \\ d \equiv 1 \bmod{4}}} \mu_{\mz}(d) \tilde{g}(d)\overline{\chi}_d(l) H(d^2l, Q/d^2),
\end{equation*}
where
\begin{equation*}
H(d^2l, X) = \sum_{\substack{n \in \mz[i] \\ n \equiv 1 \bmod{(1+i)^3}}}  \frac{\overline{\chi}_{n}(d^2l) g(n)}{\sqrt{N(n)}} w\left(\frac{N(n)}{X}\right).
\end{equation*}
We estimate $H$ with the next lemma.
\begin{lemma}
\label{lemma:Hbound}
 For any $l \in \mz[i]$, write $l = l_0 l_1$ where $l_0$ is a unit times a power of $1+i$, and $l_1 \equiv 1 \pmod{(1+i)^3}$.  Then we have
\begin{equation*}
\label{eq:Hbound}
 H(l, X) \ll X^{1/2 + \varepsilon} N(l_1)^{1/4} + X^{3/4} N(l_1)^{1/8 + \varepsilon}.
\end{equation*}
\end{lemma}
\begin{proof}
 Writing $l = l_0 l_1$ as above, we set
\begin{equation*}
  \overline{\leg{l}{n}}_4 = \overline{\leg {l_1}{n}}_4 \cdot \overline{\leg{l_0}{n}}_4.
\end{equation*}
From the discussion in Section \ref{sec2.4}, the function $\lambda(n) = \overline{\leg{l_0}{n}}_4$ is a Hecke character $\pmod{16}$ of trivial infinite type.
Thus
\begin{equation*}
H(l,X) = \sum_{\substack{n \in \mz[i] \\ n \equiv 1 \bmod{(1+i)^3}}}  \frac{\lambda(n) \overline{\leg {l_1}{n}}_4 g(n)}{\sqrt{N(n)}} w\left(\frac{N(n)}{X}\right).
\end{equation*}
Note that the identity \eqref{eq:gmult} implies $\overline{\leg {l_1}{n}}_4 g(n) = g(l_1, n)$ for $(n, l_1) = 1$.  Upon introducing the Mellin transform of $w$, we get
\begin{equation}
\label{eq:HlX}
 H(l,X) =
\frac{1}{2 \pi i} \int\limits_{(2)} \widetilde{w}(s) X^s h( l_1, 1/2 + s;\lambda) \dif s,
\end{equation}
   where $\widetilde{w}(s)$ is defined as in \eqref{w}. \newline

  We move the line of integration in \eqref{eq:HlX} to $\Re(s) = 1/2 + \varepsilon$, crossing a pole at $s =3/4$, which contributes
\begin{equation*}
\ll X^{3/4} N(l_1)^{1/8 + \varepsilon}.
\end{equation*}
The contribution from the new line of integration is
\begin{equation*}
\ll  X^{1/2+\varepsilon} N(l_1)^{1/4}.
\end{equation*}
This completes the proof of Lemma \ref{lemma:Hbound}.
\end{proof}

  Now, to prove Lemma \ref{lemma:sumofGauss}, we treat $|d| \leq Y$ and $|d| > Y$ separately, where $Y$ is a parameter to be chosen later.  For $|d| \leq Y$ we use Lemma \ref{lemma:Hbound}, while for $|d| > Y$ we use the trivial bound $H(l,X) \ll X$.  Thus
\begin{equation*}
 H'(l,Q) \ll \sum_{|d| \leq Y} \leg{Q}{d^2}^{1/2+\varepsilon} N(d^2l)^{1/4} + \sum_{|d| \leq Y} \leg{Q}{d^2}^{3/4} N(d^2l)^{1/8+\varepsilon} + \sum_{|d| > Y} \frac{Q}{d^2},
\end{equation*}
which simplifies as
\begin{equation*}
 H'(l,Q) \ll Q^{1/2 + \varepsilon} Y^{1-2\varepsilon} N(l)^{1/4} + Q Y^{-1} + Q^{3/4} N(l)^{1/8 + \varepsilon}Y^{\varepsilon}.
\end{equation*}
The optimal choice of $Y$ is $Y = Q^{1/4} N(l)^{-1/8}$ and gives Lemma \ref{lemma:sumofGauss}. Applying this lemma and summing trivially over $m$ in the expressions for $M^+_{2,1}$ and $M^+_{2,2}$, one easily deduces \eqref{eq:M2bound}.  This completes the proof of Lemma \ref{lem2}.  \newline

\noindent{\bf Acknowledgments.} P. G. is supported in part by NSFC grant 11871082 and L. Z. by the FRG grant PS43707 and the Faculty Goldstar Award PS53450. Parts of this work were done when P. G. visited the University of New South Wales (UNSW). He wishes to thank UNSW for the invitation, financial support and warm hospitality during his pleasant stay.

\bibliography{biblio}
\bibliographystyle{amsxport}

\vspace*{.5cm}

\noindent\begin{tabular}{p{8cm}p{8cm}}
School of Mathematical Sciences & School of Mathematics and Statistics \\
Beihang University & University of New South Wales \\
Beijing 100191 China & Sydney NSW 2052 Australia \\
Email: {\tt penggao@buaa.edu.cn} & Email: {\tt l.zhao@unsw.edu.au} \\
\end{tabular}

\end{document}